\title{Random Simplicial Complexes - Around the Phase Transition}
\author{Nathan Linial\thanks{Department of Computer Science, Hebrew University, Jerusalem 91904,
    Israel. e-mail: nati@cs.huji.ac.il~. Supported by ERC grant 339096 "High-dimensional combinatorics".} \and  {Yuval Peled\thanks{Department of Computer Science, Hebrew University, Jerusalem 91904,
    Israel. e-mail: yuvalp@cs.huji.ac.il~. Yuval Peled is grateful to the Azrieli Foundation for the award of an Azrieli Fellowship.}
}}

\date{In memory of Ji\v{r}\'{i} Matou\v{s}ek}

\documentclass[11pt]{article}
\usepackage{amssymb,fullpage}
\usepackage{xspace}
\usepackage{latexsym}
\usepackage{times}
\usepackage{amsfonts}
\usepackage{amsmath}

\usepackage{mathrsfs}
\usepackage{bbm}

\usepackage{verbatim}

\usepackage{amsthm}
\usepackage{amsmath}
\usepackage{amssymb}
\usepackage{graphicx}
\usepackage{subfig}
\usepackage{epstopdf}
\usepackage{titling}

\usepackage{enumerate}

\usepackage{algpseudocode}
\usepackage{algorithm}
\usepackage{algorithmicx}
\usepackage{tikz}

\newcommand{\ignore}[1]{}

\newcommand{\mR}{\ensuremath{\mathbb R}}

\newcommand{\E}{\ensuremath{\mathbb E}}
\newcommand{\C}{\ensuremath{\mathbb C}}

\newcommand{\tl}[1]{\tilde{{#1}}}
\newcommand{\I}{\ensuremath{[0,1]}}
\newcommand{\inpr}[1]{\ensuremath{\left\langle #1 \right\rangle}}
\newcommand{\LMvarc}[1]{\ensuremath{Y_d\left(n,\frac {#1}{n}\right)}}
\def \B {\mathcal B}
\def \A {\mathcal A}
\def \D {\mathcal D}

\def \cL {\mathcal L}
\def \R {\mathcal R}
\def \P {\mathcal P}
\def \M {\mathcal M}

\newtheorem{theorem}{Theorem}[section]

\newtheorem{lemma}[theorem]{Lemma}
\newtheorem{claim}[theorem]{Claim}

\def \dim {{\rm dim}}

\def \ker {{\rm ker}}

\def \LMc {Y_d\left(n,\frac cn\right)}
\def \ccol {\gamma_d}
\def \cacy {c_d}
\def \SH {\mbox{SH}}

\begin{document}
\maketitle

\section*{Abstract}

This article surveys some of the work done in recent years on random simplicial complexes. We mostly consider higher-dimensional analogs of the well known phase transition in $G(n,p)$ theory that occurs at $p=\frac 1n$. Our main objective is to provide a more streamlined and unified perspective of some of the papers in this area.

\section{Introduction}
There are at least two different perspectives from which our subject can be viewed. We survey some recent developments in the emerging field of {\em high-dimensional combinatorics}. However, these results can be viewed as well as part of an ongoing effort to apply {\em the probabilistic method in topology}. The systematic study of random graphs was started by Erd\H{o}s and  R\'enyi in the early 1960's and had a major impact on discrete mathematics, computer science and engineering. Since graphs are one-dimensional simplicial complexes, why not develop an analogous theory of $d$-dimensional random simplicial complexes for all $d\ge 1$? To this end, an analog of Erd\H{o}s and R\'enyi's $G(n,p)$ model, called $Y_d(n,p)$, was introduced in \cite{LM}. Such a simplicial complex $Y$ is $d$-dimensional, it has $n$ vertices and has a full $(d-1)$-dimensional skeleton. Each $d$-face is placed in $Y$ independently with probability $p$. Note that $Y_1(n,p)$ is identical with $G(n,p)$.

One of the most natural questions to ask in any model of random graphs concerns graph {\em connectivity}. As Erd\H{o}s and R\'enyi famously showed, the threshold for graph connectivity in $G(n,p)$ is $p=\frac{\ln n}{n}$. To draw the analogy from a topological perspective, one should seek the threshold for the vanishing of the $(d-1)$-st homology. This indeed was the motivating problem in~\cite{LM}. As that paper showed, and together with subsequent work~\cite{MW} this threshold in $Y_d(n,p)$ is $p=\frac{d\ln n}{n}$. Here the coefficients can come from any fixed finite Abelian group. The same question for {\em integral} $(d-1)$-st homology has attracted considerable attention and the answer is believed to be the same. This was recently confirmed for $d=2$ ~\cite{LuP}, and is not yet fully resolved for higher dimensions (but see~\cite{HKP}). The threshold for the vanishing of the fundamental group of $Y_2(n,p)$ is fairly well (but still not perfectly) understood~\cite{BHK,KPS}.

Since we tend to work by analogy with the $G(n,p)$ theory, it is a very challenging problem to seek a high-dimensional counterpart to the {\em phase transition} that occurs at $p=\frac 1n$. It is here that the random graph asymptotically-almost-surely (a.a.s.) acquires cycles. Namely, for every $1>c>0$ there is a $1>q=q(c)>0$ such that a graph in $G(n,\frac cn)$ is a forest with probability $q+o_n(1)$, but for $p\ge \frac 1n$, a $G(n,p)$ graph has, a.a.s., at least one cycle. These notions have natural analogs in higher-dimensional complexes that suggest what is being sought. However, even more famously, a {\em giant} connected component with $\Omega(n)$ vertices emerges at $p=\frac 1n$. Since there is no natural notion of connected components at dimensions $d>1$, it is not even clear what to ask. Finding the correct framework for asking this question and discovering the answer is indeed one of the main accomplishments of the research that we survey here.

Another reason that makes the high-dimensional scenario more complicated than the graph-theoretic picture is that there are several natural analogs for acyclicity. A $(d-1)$-face $\tau$ in a $d$-complex $Y$ is {\it free} if it is contained in exactly one $d$-dimensional face $\sigma$ of $Y$. In the corresponding {\em elementary collapse} step\footnote{Recall that an elementary collapse is a homotopy equivalence.}, $\tau$ and $\sigma$ are removed from $Y$. We say that $Y$ is {\em $d$-collapsible} if it is possible to eliminate all its $d$-faces by a series of elementary collapses. Otherwise, the maximal subcomplex of $Y$ in which all $(d-1)$-faces are contained in at least $2$ $d$-faces is called the {\em core} of $Y$. Clearly a graph (i.e., a $1$-dimensional complex) is $1$-collapsible if and only if it is acyclic, i.e., a forest.

A $d$-complex $Y$ is said to be {\em $d$-acyclic} if its $d$-th homology group vanishes. Namely, if the $d$-dimensional boundary matrix $\partial_d(Y)$ has a trivial right kernel. Unless otherwise stated, we consider this matrix over the reals. The real $d$-Betti number of $Y$ is $\beta_d(Y;\mR):=\dim H_d(Y;\mR) = \dim( \ker\partial_d(Y))$. 

Whereas acyclicity and $1$-collapsibility are equivalent for graphs, this is no longer the case for $d$-dimensional complexes\footnote{We briefly refer to a $d$-dimensional complexes as a $d$-complex}. Clearly, a $d$-collapsible simplicial complex has a trivial $d$-th homology, but the reverse implication does not hold in dimension $d\ge 2$.

In this view, there are now two potentially separate thresholds to determine in $Y_d(n,p)$: For $d$-collapsibility and for the vanishing of the $d$-th homology. These questions were answered and the respective thresholds were determined in a series of four papers. A lower bound on the threshold for $d$-collapsibility was found in~\cite{ALLM} and a matching upper bound was proved in~\cite{col2}\footnote{In the present article we note an error in that paper, but also indicate how to overcome it. In doing so we also derive some additional information on critical complexes.} . An upper bound on the threshold for the vanishing of the $d$-th homology was found in~\cite{acyc}, with a recent matching lower bound for real homology~\cite{LP}. We conjecture that the same bound holds for all coefficient rings, but this remains open at present.

The purpose of this paper is to survey these results and present the main ingredients of the proofs. In particular, we highlight the key role of the local structure of random complexes in all these proofs.

Both thresholds are of the form $p=\frac cn$. Namely there is a constant $c=\ccol$ corresponding to the $d$-collapsibility threshold and $c=\cacy$ for acyclicity. As functions of the dimension $d$, the constants $\ccol$ and $\cacy$ differ substantially. Our results allow us to numerically compute them to desirable accuracy (See Table \ref{tbl:thrVals}).

\begin{table}
\begin{center}
\begin{tabular}{| l | c  | c | c | c |c|c|c| }
  \hline
  $d$ & \bf{2} & \textbf{3} & \textbf{4} & \textbf{5} & \textbf{10}&\textbf{100}&\textbf{1000} \\
  \hline
  $\ccol$ &
$2.455$ &
$3.089$ &
$3.509$ &
$3.822$ &
$4.749$ &
$7.555$ &
$10.175$ \\
  \hline
  $\cacy$ &
$2.754$&
$3.907$&
$4.962$&
$5.984$&
$11-10^{-3.73}$&
$101-10^{-41.8}$&
$1001-10^{-431.7}$\\
  \hline
\end{tabular}

\caption{The critical constants $\ccol$ and $\cacy$.}\label{tbl:thrVals}
\end{center}
\end{table}

Before stating the theorems, a small technical remark is in order. An obvious obstacle for a $d$-complex $Y$ to be either $d$-collapsible or $d$-acyclic is that it contains the boundary of a $(d+1)$-simplex $\partial\Delta_{d+1}$, i.e. all the $d+2$ $d$-faces that are spanned by some $d+2$ vertices. In the random complex $\LMc$, these objects appear with probability bounded away from both zero and one, and it is easy to see that their number is Poisson distributed with a constant expectation. In particular, $\LMc$ is $\partial\Delta_{d+1}$-free with positive probability. There are several ways to go around this technical difficulty. In \cite{ALLM} a model of random complexes conditioned on being $\partial\Delta_{d+1}$-free was considered, which allowed a cleaner form for the theorems. Here we work with the simple binomial model, and consequently must mention these simplices.

We turn to the main theorems. Let $d\ge 2$ be an integer, $c>0$ real and denote the core of $Y=\LMc$ by $\tilde Y$. We define $\ccol$ as the minimum of the function $\psi(x):=-\frac{\ln{x}}{(1-x)^d}~,~~0<x<1$. Furthermore, we let $x_*$ be the unique root in $(0,1)$ of
$
(d+1)(1-x)+(1+dx)\ln x=0,
$
and $\cacy:=\psi(x_*)$.

In addition, for an integer $k$ and $\lambda>0$ real, we let $\Psi_k(\lambda):=\Pr[\mbox{Poi}(\lambda)\ge k]$, and $t=t(c,d)$ be the smallest positive root of $t=e^{-c(1-t)^d}$, or equivalently $1-t = \Psi_1(c(1-t)^d)$.

\begin{theorem}
\label{thm:main}
Let $d\ge 2$ be an integer, $c>0$ real, and $Y=\LMc$.
\begin{enumerate}[(I)]
\item The collapsible regime: If $c<\ccol$ then a.a.s.\ either $Y$ is $d$-collapsible or its core is comprised of $O_d(1)$ vertex disjoint $\partial\Delta_{d+1}$'s.
\item The intermediate regime: If $\ccol<c<\cacy$ then a.a.s.\ 

\begin{enumerate}
 \item $Y$ is not $d$-collapsible. Moreover, its core contains a constant fraction of the $(d-1)$-faces:
\begin{equation}
\label{eqn:fcnms}
f_{d-1}(\tilde Y) = \Psi_2(c(1-t)^d)\binom nd(1+o(1))~~~~,~~~~
f_d(\tilde Y) = \frac cn \binom n{d+1}(1-t)^{d+1}(1+o(1)).
\end{equation}
In particular, $f_{d}(\tilde Y)<f_{d-1}(\tilde Y)$.
\item Either $Y$ is $d$-acyclic or $H_d(Y;\mR)$ is generated by $O_d(1)$ vertex disjoint $\partial\Delta_{d+1}$'s.
\end{enumerate}
\item The cyclic regime: If $c>\cacy$ then a.a.s.\ $H_d(Y;\mR)$ is non-trivial. Furthermore, $f_{d-1}(\tilde Y)$ and $f_{d}(\tilde Y)$ still satisfy equation (\ref{eqn:fcnms}), but in this regime $f_{d-1}(\tilde Y)<f_{d}(\tilde Y)$ and
\[
\beta_d(Y) =  \left(\frac{c}{d+1}(1-t)^{d+1}-(1-t)+ct(1-t)^d \right){n \choose d}(1+o(1))=\big(f_{d}(\tilde Y) - f_{d-1}(\tilde Y)\big)(1+o(1)).
\]
\end{enumerate}
\end{theorem}
\begin{figure}[h!]
\label{fig:thm1}
  \centering
     \includegraphics[width=0.7\textwidth]{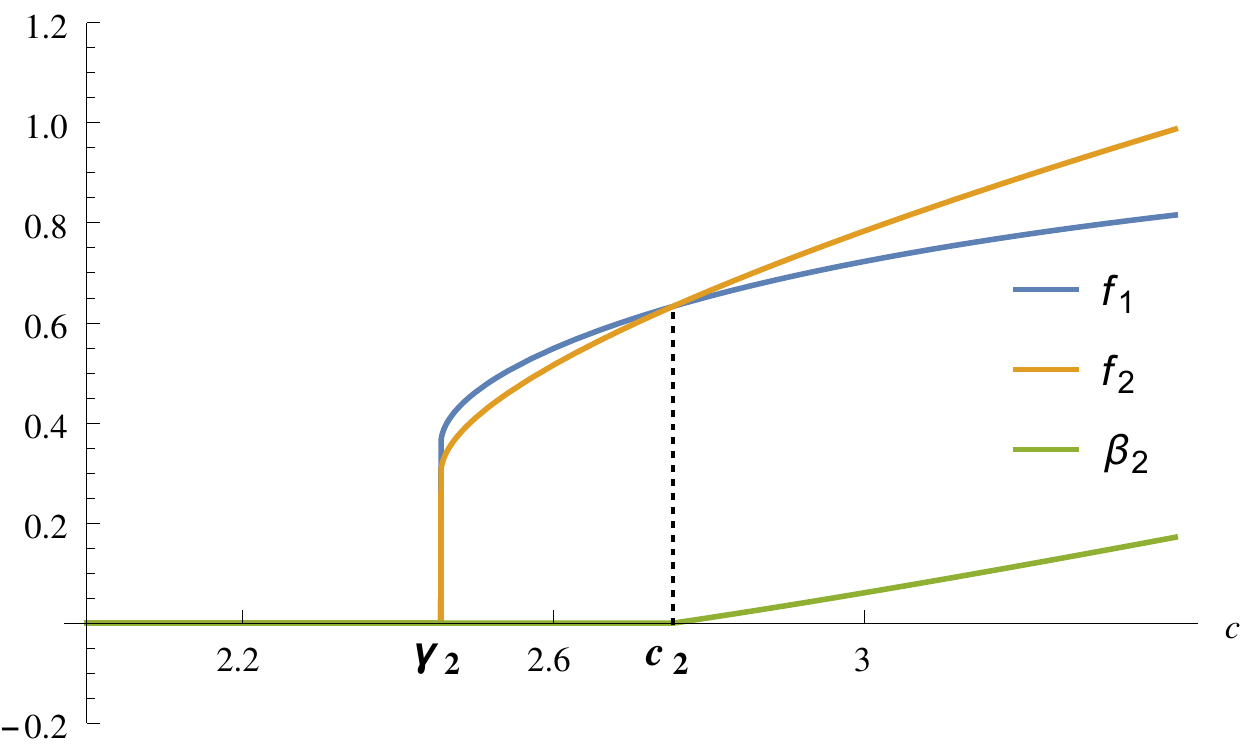}
    \caption{Illustration of Theorem \ref{thm:main} for $d=2$. Here $f_1,f_2$ are the face numbers of $\tilde Y$ and $\beta_2$ is the second Betti number. The functions are normalized by $\binom n2$, and $n\to\infty$.}
       \end{figure}
It is not hard to determine the asymptotic behaviour (in $d$) of these expressions, namely,
\[\cacy=(d+1)(1-e^{-(d+1)})+O_d(d^3e^{-2d})\]
and
\[\ccol=(1+o_d(1))\ln d.\]
Consequently, there is a wide range of the parameter $p=p(d,n)$ for which almost all the complexes in $Y_d(n,p)$ are acyclic and non-collapsible. There are strong indications that complexes from this range have interesting and unexpected properties.

Note that $f_d(\tilde Y) - f_{d-1}(\tilde Y)>0$ implies that $H_d(Y;\mR)\neq 0$. Moreover, this difference of the face numbers is a lower bound for the $d$-th Betti number. Theorem \ref{thm:main} shows that for all $0\le p\le 1$, and up to the appearance of $\partial\Delta_{d+1}$'s these two conditions are typically equivalent and the lower bound is asymptotically tight. This is clearly a probabilistic statement which does not hold in general.

We turn to deal with the emergence of the giant component, a subject on which there exists an extensive body of literature. As mentioned above, there is no obvious high-dimensional counterpart to the notion of connected components, and we need a conceptual idea in order to even get started. The notion of {\em shadows}, introduced in \cite{LNPR}, offers a way around this difficulty. The idea is to tie connected components with cycles, which do have natural high dimensional counterparts. The shadow of a graph $G$ is the set of those edges that are not in $G$, whose addition to $G$ creates a new cycle. It turns out that the giant component emerges exactly when the shadow of the evolving random graph acquires positive density. In particular, for $c>1$ the shadow of $G(n,\frac cn)$ has density $(1-t)^2+o(1)$, where $t$ is the unique root in $(0,1)$ of $t=e^{-c(1-t)}$ (See Figure \ref{fig:shadowGraph}).

This suggests how we should define $\SH_\mR(Y)$, the shadow of $Y$, a $d$-dimensional complex with full $(d-1)$-skeleton. Namely, it is the following set of $d$-faces:
\[
\SH_\mR(Y)=\{\sigma\notin Y:H_d(Y;\mR)\text{~is a proper subspace of~} H_d(Y\cup\{\sigma\};\mR)\}.
\]
In words, a $d$-face belongs to $\SH_\mR(Y)$ if it is not in $Y$ and its addition to $Y$ creates a new $d$-cycle.

We are considering throughout the vanishing of the $d$-th homology and $d$-collapsibility. These two notions capture acyclicity from an algebraic resp. combinatorial perspective. For $d=1$ the two coincide, but they differ widely for $d\ge 2$. This dual perspective carries over to two notions of shadows. A $d$-face $\sigma$ that does not belong to a $d$-complex $Y$ is in $Y$'s $\mR$-shadow if its addition to $Y$ increases the $d$-homology. It is in the {\em C-shadow} of $Y$ if its addition to $Y$ increases the core. Again, these notions coincide for $d=1$, and the $\mR$-shadow is always contained in the C-shadow, but for $d\ge 2$ they may differ.

The notion of shadows lets us compare the phase transitions of random graphs and random complexes of higher dimensions, and a substantial qualitative difference reveals itself. While the density of the shadow of $G(n,p)$ undergoes a smooth transition around $p=1/n$, when $d\geq 2$ both the C-shadow and the $\mR$-shadow of $\LMc$ undergo {\em discontinuous} first-order phase transitions at the critical points  $\ccol$ and $\cacy$ respectively.
\begin{theorem}
\label{thm:randShadow}
Let $d\ge 2$ be an integer, $c>0$ real, and $Y=\LMc$.
\begin{enumerate}[(I)]
\item The collapsible regime: If $c<\ccol$ then a.a.s.\ $|\SH_\mR(Y)| \le |\SH_C(Y)|=\Theta(n).$

\item The intermediate regime: If $\ccol<c<\cacy$ then a.a.s.\
$|\SH_\mR(Y)| = \Theta(n)$, and  $$|\SH_C(Y)| = \binom n{d+1}((1-t)^{d+1}+o(1)).$$

\item The cyclic regime: If $c>\cacy$ then a.a.s.\ the size of both $\SH_\mR(Y)$ and $\SH_C(Y)$ is $\binom n{d+1}((1-t)^{d+1}+o(1)).$
\end{enumerate}
\end{theorem}
\begin{figure}[h!]
    \centering
    \subfloat[~Density of the shadow of $G(n,\frac cn)$.]{{\includegraphics[width=0.465\textwidth]{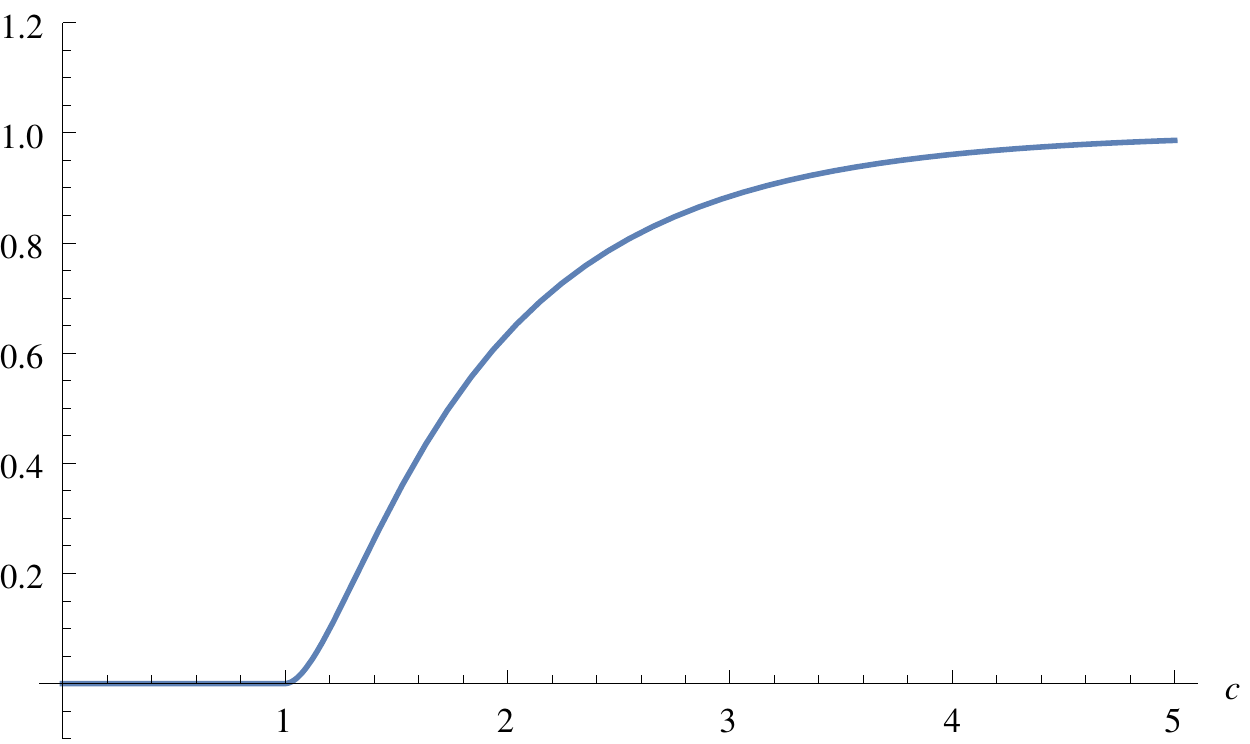} }\label{fig:shadowGraph}}%
    \qquad
    \subfloat[~Density of the C-shadow and $\mR$-shadow of $Y_2\left(n,\frac cn\right)$.]{{\includegraphics[width=0.465\textwidth]{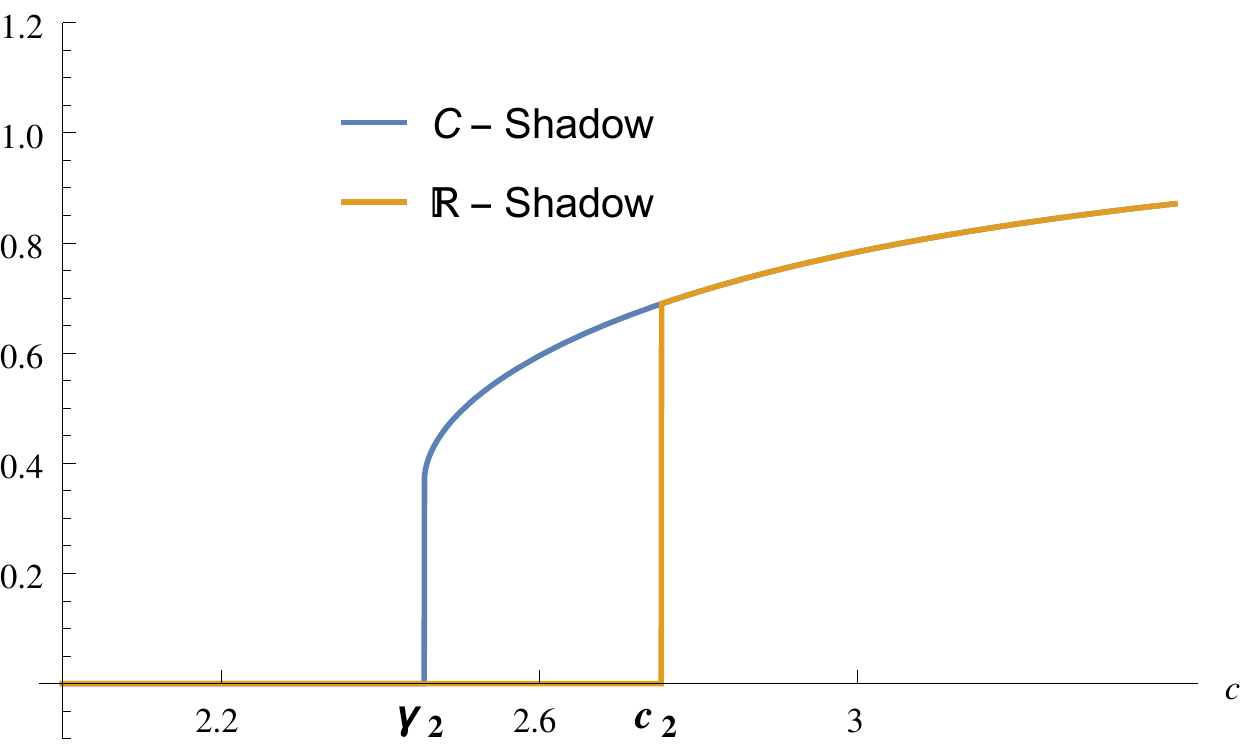} }\label{fig:shadowCompelx}}%
    \caption{Illustration of Theorem \ref{thm:randShadow} for $d=2$, and comparison to the density of the shadow of a random graph. }%
    \label{fig:shadowComparison}%
\end{figure}

%
%

An essential idea that is common to all these results is that in the range $p=\Theta(\frac 1n)$ many of the interesting properties of $Y_d(n,p)$ can be revealed by studying its {\em local} structure. Initially, this seemed as merely a useful tool in studying the threshold for $d$-collapsibility, and in establishing an upper bound on the threshold of the vanishing of the $d$-th homology. However, in obtaining a lower bound on this threshold, it became apparent that this idea should be viewed in the wider context of {\em local weak limits}. This framework was introduced by Benjamini and Schramm \cite{ben_sch} and Aldous and Steele \cite{ald_ste}. In recent years, this approach was used in deriving new asymptotic results in various fields of mathematics (e.g. \cite{ald_zeta,lyo_asym}).

The study of $d$-collapsibility in random complexes was significantly influenced by work on $k$-cores in random hypergraphs and specifically the works by Molloy \cite{molloy} and Riordan \cite{rio}. Also, the proof of Theorem \ref{thm:main} makes substantial use of tools from the paper of Bordenave, Lelarge and Salez \cite{diluted} on the rank of the adjacency matrix of random graphs.

The rest of the paper is organized as follows. Section \ref{sec:prel} gives some necessary background material about simplicial complexes. In Section \ref{sec:local} we introduce the concept of a Poisson $d$-tree which is the local weak limit of random simplicial complexes. 
The main ingredients of the proofs of the main theorems are presented in Sections \ref{sec:col} and \ref{sec:acy}, that respectively address the subjects of collapsibility and acyclicity. Concluding remarks and open questions are presented in Section \ref{sec:open}.

\section{Preliminaries}
\label{sec:prel}

A simplicial complex $Y$ is a collection of subsets of its {\em vertex set} $V$ that is closed under taking subsets. Namely, $\sigma\in Y$ and $\tau\subseteq \sigma$ imply that $\tau\in Y$ as well. Members of $Y$ are called {\em faces} or {\em simplices}. The {\em dimension} of the simplex $\sigma\in Y$ is defined as $|\sigma|-1$, and $\dim(Y)$ is defined as $\max\dim(A)$ over all faces $A\in Y$. A $d$-dimensional simplex is also called a $d$-simplex or a $d$-face, and a $d$-dimensional simplicial complex is also referred to as a $d$-complex. The set of $j$-faces in $Y$ is denoted by $Y_j$, and the face numbers by $f_j(Y):=|Y_j|$. For $t<\dim(Y)$, the $t$-{\em skeleton} of $Y$ is the simplicial complex that consists of all faces of dimension $\le t$ in $Y$, and $Y$ is said to have a {\em full} $t$-dimensional skeleton if its $t$-skeleton contains all the $t$-faces of $V$. In this paper, the {\em degree} $d_Y(\tau)$ of a face $\tau$ in a complex $Y$ is the number of $\dim(Y)$-faces that contain it. A face of degree zero is said to be {\em exposed}. Although we are directly interested only in finite complexes, infinite ones do play a role here, but we consider only {\em locally-finite} complexes in which every face has a finite degree. We occasionally use the bipartite incidence graph between $(d-1)$-faces and $d$-faces of a $d$-complex $Y$. This allows us, in particular, to speak about {\em distances} among such faces.

The permutations on the vertices of a face $\sigma$ are split in two
{\em orientations}, according to the permutation's sign.
The {\em boundary operator} $\partial=\partial_d$ maps an oriented
$d$-simplex $\sigma = (v_0,...,v_d)$ to the formal sum
$\sum_{i=0}^{d}(-1)^i(\sigma^i)$, where $\sigma^i=(v_0,...v_{i-1},v_{i+1},...,v_d)$ is an oriented
$(d-1)$-simplex. We fix some commutative ring $R$ and linearly extend
the boundary operator to free $R$-sums of
simplices. We denote by $\partial_d(Y)$ the $d$-dimensional boundary operator of a $d$-complex $Y$.

When $Y$ is finite, we consider the ${f_{d-1}(Y)}\times{f_d(Y)}$ matrix form of $\partial_d$ by choosing arbitrary orientations for
$(d-1)$-simplices and $d$-simplices. Note that changing the
orientation of a $d$-simplex (resp. $d-1 $-simplex) results in
multiplying the corresponding column (resp. row) by $-1$.

The $d$-th homology group $H_d(Y;R)$ (or vector space if $R$ is a field) of a $d$-complex $Y$ is the (right) kernel of its boundary operator $\partial_d$. Most of the homology groups in this paper are considered over $\mR$. An element in $H_d(Y;\mR)$ is called a {\em $d$-cycle}, and the whole group is called the $d$-cycle space of $Y$. The {\em $d$-th Betti number} $\beta_d(Y;\mR)$ of a complex $Y$ is defined to be the dimension of $H_d(Y;\mR)$. 

Recall the concept of elementary collapse as defined in the introduction. A $d$-collapse phase is a procedure in which all the possible elementary $d$-collapses take place at once. In case more than one $(d-1)$-face can collapse some $d$-face, one of them is chosen arbitrarily. Given a $d$-complex $Y$, the complex $R_k(Y)$ is the complex that is obtained from $Y$ after $k$ phases of $d$-collapse. Similarly, $R_{\infty}(Y)$ is obtained after all possible $d$-collapse steps are carried out. A {\em $d$-core} (or core, for brevity) is a $d$-complex in which all the $(d-1)$-faces are of degree $\ge 2$. The {\em $d$-core of $Y$} is the maximal $d$-core subcomplex of $Y$. Note that the $d$-core of $Y$ is obtained from $R_{\infty}(Y)$ by removing the exposed $(d-1)$-faces.

\section{Poisson $d$-tree}
\label{sec:local}
The concept of a Poisson $d$-tree process was introduced in \cite{ALLM} and turned out to be extremely useful in the study of random simplicial complexes. It can be viewed as a high-dimensional counterpart of the Poisson Galton-Watson process which plays a key role in the study of the giant component in $G(n,p)$ graphs.

A {\em rooted $d$-tree} is a pair $(T,o)$ where $T$ is a $d$-complex and $o$ is some $(d-1)$-face of $T$. A $d$-tree is generated by the following process. Initially the complex consists of the $(d-1)$-face $o$. At every step $k\ge 0$, every $(d-1)$-face $\tau$ of distance $k$ from $o$ picks a non-negative number $m=m_\tau$ of new vertices $v_1,...,v_m$, and adds the $d$-faces $v_1\tau,...,v_m\tau$ to $T$.

We use some self-explanatory terminology in our study of $d$-trees. A {\em leaf} is a $(d-1)$-face with no {\em descendant} $d$-faces. A $(d-1)$-face $\tau$ is an {\em ancestor} of a $(d-1)$-face $\tau'$ if $\tau'$ belongs to the {\em subtree rooted at $\tau$}. If $(T,o)$ is a rooted $d$-tree and $T'$ is a subtree of $T$ which contains the $(d-1)$-face $o$, we refer to $(T',o)$ as a {\em rooted subtree} of $(T,o)$. The {\em depth} of a $(d-1)$-face is its distance from the root, and the depth of the $d$-tree is the maximal depth of any of its $(d-1)$-faces.

A {\em Poisson $d$-tree with parameter $c$}, denoted by $T_d(c)$, is a rooted $d$-tree in which all the numbers $m_\tau$ throughout this generative process are i.i.d.\ $Poi(c)$-distributed. The rooted subtree of $T_d(c)$ that consists of the first $k$ generations of this process is denoted $T_{d,k}(c)$.

The most important fact about $T_d(c)$ in this context is that it approximates the local neighborhood of a $(d-1)$-face in $\LMc$.
\begin{lemma}\cite{ALLM}
\label{lem:loc_app}
For every fixed integer $k>0$, the $k$-neighborhood of a fixed $(d-1)$-face {$\tau$} in { $Y=\LMc$} converges in distribution to the $k$-neighborhood of the root of $T_d(c)$ as $n\to\infty$.
\end{lemma}
\begin{proof}
First, we observe that the degree of every $(d-1)$-face in $Y$ is $\mbox{Bin}\left(n-d,\frac cn\right)$-distributed. Chernoff's Inequality implies that a.a.s.\ no degree exceeds $(d+1)\log{n}$. We claim that for every fixed $k\ge 0$, the $k$-neighborhood of $\tau$ in $Y$ is a.a.s.\ a $d$-tree. The cases $k=0$ and $k=1$ are trivial. Assume, by induction, that the $k$-neighborhood $N_k(\tau)$ of $\tau$ is a $d$-tree. The tree structure will be violated in the $(k+1)$-th layer if and only if there is some $(d-1)$-face $\eta$ of distance $k$ from $\tau$, and some vertex $v\in N_k(\tau)$ such that $\eta v \in Y$. However, by the bound on the degrees of the $(d-1)$-faces, there only $O(\log^k{n})$ such $(d-1)$-faces $\eta$ and  $O(\log^{k}n)$ such vertices $v$. Therefore, the probability that a face of the form $\eta v$ belongs to $Y$ is negligible. In addition, conditioned on the $(k+1)$-neighborhood being a $d$-tree, the number $m_\eta$ of new vertices that $\eta$ adds to the $d$-tree is $\mbox{Bin}\left(n-o(n),\frac cn\right)$-distributed, which tends to Poisson with parameter $c$ as $n\to\infty$.
\end{proof}
This lemma easily implies convergence  of $\LMc\to T_d(c)$ in the sense of local weak convergence introduced by Benjamini and Schramm \cite{ben_sch} and Aldous and Steele \cite{ald_ste}. Here is a brief explanation of this concept. A rooted $d$-complex is a pair $(Y,\tau)$ of a $d$-complex and some $(d-1)$-face in it. We denote by $(Y,\tau)_k$ the $\tau$-rooted subcomplex of $(Y,\tau)$ comprised of all the $d$-faces of distance at most $k$ from $\tau$ and their subfaces. Let $\mathcal{Y}_d$ be the set of all (isomorphism types of) rooted $d$-complexes, equipped with the metric
\[
\mbox{dist}((Y,\tau),(Y',\tau')) = \inf\left\{\frac{1}{t+1}~:~(Y,\tau)_k\cong(Y',\tau')_k\right\}.
\]
It can be easily verified that $(\mathcal{Y}_d,\mbox{dist})$ is a separable and complete metric space, which comes, as usual, equipped with its Borel $\sigma$-algebra (See \cite{ald_lyo}).
The fact that $\LMc$ converges to $T_d(c)$ means that for every bounded and continuous function $f:\mathcal{Y}_d\to\mR$,
\[
\E_{Y\sim\LMc}[f(Y,\tau)]\xrightarrow[n\to\infty]{} \E_{T\sim T_d(c)}[f(T,o)].
\]
As we explain below, this fact will be applied directly to a function of particular interest in this context, namely, the degree of the root $\tau$ after $k$ phases of $\tau$-rooted collapse. In addition, it will be used in combination with the spectral theorem to bound the Betti numbers of $\LMc$ with the spectral measure of the Poisson $d$-tree.

\subsection{Rooted collapse}
Let $Y$ be a $d$-complex and $\tau$ some $(d-1)$-face of $Y$. A {\em $\tau$-rooted collapse} of $Y$ is a $d$-collapse process in which we forbid to collapse $\tau$.
Let $k$ be a non-negative integer. The complex obtained from $Y$ after $k$ phases in the $\tau$-rooted collapse process is denoted by $R_k(Y,\tau)$.

In the case $Y=\LMc$, the degree $d_{R_k(Y,\tau)}(\tau)$ turns out to be relevant to several different questions. We approximate it using $\delta_k:=d_{R_k(T,o)}(o)$, where $T=T_d(c)$, the Poisson $d$-tree with root $o$.

\begin{lemma}
With the above notations
\[\E[d_{R_k(Y,\tau)}(\tau)] \xrightarrow[n\to\infty]{} \E[\delta_k].
\]
\end{lemma}
Note that this is not a direct corollary of the local weak convergence. Even though the function $d_{R_k(Y,\tau)}(\tau)$ is continuous, being dependent only on some fixed neighborhood of the root, it is not bounded. Nevertheless, we allow ourselves to omit the proof, since this difficulty can be bypassed by a simple calculus trick. Namely, by considering the function $\min\{d_{R_k(Y,\tau)}(\tau),A\}$, where $A$ is a sufficiently
large constant.

\begin{lemma}
\label{lem:delta_k_distr}
Let $c>0$ and $(t_k)_{k\ge -1}~$ a sequence of real numbers defined by $$t_{-1}=0~~,~~~t_{k+1}=e^{-c(1-t_k)^d},~\forall k\ge 0.$$ Then, $\delta_k$ is Poisson distributed with parameter $c(1-t_{k-1})^d$, for every $k\ge 0$.
\end{lemma}
We refer throughout the paper to the sequences $t_k$ of real numbers and $\delta_k$ of random variables that are defined here without denoting the underlying parameter $c>0$ that is clear from the context.
\begin{proof}
By induction on $k$. The case $k=0$ is trivial since $\delta_0$ is Poisson distributed with parameter $c$. For the induction step, let us consider the distribution of $\delta_k$. A $d$-face $\sigma$ that contains the root $o$ survives $k$ phases of rooted collapse if and only if each of its $(d-1)$-faces $\tau$ other than the root (there are $d$ such $\tau$'s) is contained in a $d$-face other then $\sigma$ after $k-1$ phases. This occurs if and only if $\tau$ has a positive degree in the subtree of $T$ rooted at $\tau$ after $k-1$ phases of $\tau$-rooted collapse. Since this subtree is also a Poisson $d$-tree with parameter $c$, this occurs, by the induction hypothesis, with probability $\Pr[\delta_{k-1} > 0]=1-t_{k-1}$. Moreover, different branches of the tree are independent, so these events for different $\tau$'s and $\sigma$'s are independent.
Namely, the distribution of $\delta_k$ is a Binomial distribution with $\delta_0\sim\mbox{Poi(c)}$ trials and success probability $(1-t_{k-1})^d$. By a standard computation in probability theory, this implies that $\delta_k$ is Poisson distributed with parameter $c(1-t_{k-1})^d$.
\end{proof}

We say that {\em a rooted $d$-tree is collapsible} if its root gets exposed in the rooted collapse process. For instance, the previous lemma shows that the probability that $T_d(c)$ is collapsed after $k$ phases is $t_{k}$, and the probability that $T_d(c)$ is collapsible is $t=t(c,d)$.

\section{$d$-collapsibility}
A simple calculus exercise tells us that the behavior of the sequence $(t_k)$ of Lemma \ref{lem:delta_k_distr} changes quite substantially when $c=\ccol$. Namely,
\[\lim_{k\to\infty} t_k = \left\{\begin{matrix}
1&c<\ccol\\
t& c>\ccol
\end{matrix}, \right.
\]
where $t=t(c,d)$ is as defined just before Theorem \ref{thm:main}. In other words, if $c<\ccol$ then the root of $T_d(c)$ gets exposed after $k$ collapse phases with probability $1-o_k(1)$. Moreover, the expected degree of the root is $o_k(1)$. On the other hand, if $c>\ccol$ then, for arbitrarily large $k$, with probability bounded away from zero some $d$-faces that contain the root will survive the collapse process.

How do these facts reflect on the behavior of the random simplicial complex $Y=\LMc$ under $d$-collapse phases? Many parameters of $R_k(Y)$ can be understood almost directly from the $\tau$-rooted collapse of $Y$, where $\tau$ is a typical $(d-1)$-face. Moreover, the Poisson $d$-tree plays a key role here since $k$ phases of $\tau$-rooted collapse depend only on the $k$-neighborhood of $\tau$. Consequently, as $k$ grows, almost all $(d-1)$-faces of $Y$ will either collapse or become exposed in $R_k(Y)$ if $c<\ccol$. On the other hand, if $c>\ccol$, a constant fraction of the $(d-1)$-faces survive $k$ phases of $d$-collapse, but only very few of them remain with degree $1$, giving the collapse process a slim chance to continue much further. In fact, the fraction of the $(d-1)$-faces that are contained in $Y$'s core is asymptotically approximated by the probability that the root of $T_d(c)$ has degree $\ge 2$ after infinitely many collapse phases. 

While the transition from the Poisson $d$-tree to the random simplicial complex is straightforward in the subcritical regime, in the supercritical regime we follow an involved argument of Riordan \cite{rio} for $k$-cores of random graphs. The reader is encouraged to read the introduction of Riordan's paper for an intuitive discussion on the proof method.

\label{sec:col}
\subsection{The collapsible regime - Theorem \ref{thm:main}~(I)}
Let $Y=\LMc$, $c<\ccol$, and let $\tau$ be some $(d-1)$-face in $Y$.
\begin{align}
\E[f_d(R_\infty(Y))] ~\le~& \E[f_d(R_k(Y))] \nonumber\\
=~&
\frac 1{d+1}\E\left[\sum_{\tau\in Y_{d-1}} \mathbf{1}_{\tau\in R_k(Y)}\cdot d_{R_k(Y)}(\tau) \right]\nonumber\\
=~& \frac 1{d+1}{n\choose d}\E\left[\mathbf{1}_{\tau\in R_k(Y)}\cdot d_{R_k(Y)}(\tau) \right]\label{eqn:1113}\\
\le~&\frac 1{d+1}{n \choose d}\E\left[ d_{R_k(Y,\tau)}(\tau) \right]\label{eqn:1114}\\
=~&\frac 1{d+1}{n\choose d}(1+o(1))\E[\delta_k]\nonumber\\
=~&\frac{1}{d+1}{n\choose d}(1+o(1))c(1-t_{k-1})^d.\nonumber
\end{align}
Identity (\ref{eqn:1113}) is obtained by considering some fixed $(d-1)$-face $\tau$, using linearity of expectation and symmetry. 
The subsequent inequality (\ref{eqn:1114}) is due to the fact that in the $\tau$-rooted collapse process fewer collapses occur than in $d$-collapse phases, whence an inequality $d_{R_k(Y)}(\tau)\le d_{R_k(Y,\tau)}(\tau)$ between $\tau$'s degrees after $k$ phases in either $d$-collapse processes.
The following equations are straightforward applications of the lemmas in Section \ref{sec:local}.

Consequently,  $f_d(R_\infty(Y)) = o(n^d)$ a.a.s. for every $c<\ccol$.

The argument which completes the proof says that for every $c>0$, the complex $\LMc$ has no core subcomplex with $o(n^d)$ $d$-faces, other than vertex disjoint $\partial\Delta_{d+1}$'s. This is proved in Theorem 4.1 of \cite{ALLM}, concerning inclusion-minimal core complexes. It turns out that a slight modification of that proof yields a more general conclusion.

A $d$-complex whose $d$-faces are comprised of a vertex-disjoint union of boundaries of $(d+1)$-simplices is called here a $d$-{\em gravel}. 

\begin{lemma}
\label{lem:noSmallCores}
For every integer $d\ge 2$ and real $c>0$ real there is $\alpha>0$ such that a.a.s.\ the following holds. Let $Y=\LMc$, then either $R_\infty(Y)$ is a $d$-gravel or $f_d(R_\infty(Y))>\alpha n^d$.
\end{lemma}
\begin{proof}
Let $m_1:= (d^3\log{n})^d$. Our first goal is to show that every core $d$-subcomplex $C$ of $Y$ with $f_d(C)\le m_1$ is a $d$-gravel. A simple first moment argument yields that $Y$ cannot contain two intersecting copies of $\partial\Delta_{d+1}$, nor can it contain more than $\log\log n$ copies of $\partial\Delta_{d+1}$. A core $d$-complex $C$ that is comprised of exactly $l$ vertex disjoint $\partial\Delta_{d+1}$'s and $m$ additional $d$-faces is said to have {\em type} $(l,m)$.

Let $C$ have type $(l,m)$. We partition its vertex set into $S\dot{\cup}T$, where $S$ is the set of the vertices in some $\partial\Delta_{d+1}$ of $C$. Let $T'\subseteq T$ be those vertices in $T$ of degree $d+1$ in $C$ (i.e., such a vertex is in exactly $d+1$ of $C$'s $d$-faces). Since $C$ is a core, the degree of every vertex in $C$ is at least $d+1$. 

In addition, every $d$-face of $C$ with two or more vertices of degree $d+1$ is included in a $\partial\Delta_{d+1}$. Recall the notion of a {\em link of a vertex} $v$ in a simplicial complex $Y$. Namely, $\mbox{lk}_Y(v)=\{\tau \in Y:v~\dot\cup~ \tau\in Y\}$. In particular, the link of a vertex in a $d$-core is a $(d-1)$-core. Therefore, if a vertex has degree $d+1$ in $C$ then its link is a $\partial\Delta_d$. Suppose that the vertices $v,u\in C$ have degree $d+1$ in $C$ and are contained in a common $d$-face $\sigma=(u,v,x_1,...,x_{d-1})$. Their links are $\partial\Delta_d$'s so there exist vertices $u',v'\notin\sigma $ such that $\mbox{lk}_C(u)=\partial(u',v,x_1,...,x_{d-1})$ and $\mbox{lk}_C(v)=\partial(u,v',x_1,...,x_{d-1})$. This can occur only if $u'=v'$ and the claim follows. 

As a result, every non-gravel $d$-face contains at most one vertex of $T'$, so that $m\ge |T'|(d+1).$
Counting incidences of vertices and $d$-faces in $C$ yields
\[
\big(m+l(d+2)\big)\cdot(d+1) \ge (|S| + |T|)(d+1)+(|T|-|T'|).
\]
But $|S|=l(d+2)$, and a simple manipulation of these inequalities gives $|T| \le m\cdot\frac{d+3}{d+4}.$ We can assume w.l.o.g.\ that $m\le m_1,~l\le \log\log n$ and derive the following upper bound on the number of type $(l,m)$ core $d$-complexes with at most $n$ vertices
\[
n^{l(d+2)+\frac{d+3}{d+4}m}\cdot \left( |S|+|T| \right)^{(d+1)m}= n^{l(d+2)}\cdot\left[ n^{\frac{d+3}{d+4}}\cdot O(\log^{d(d+1)}{n}) \right]^m.
\]
The first term counts the choices for $S,T$, and the second the choice of non-gravel $d$-faces. We conclude that a.a.s.\ $Y$ contains no core of type $(l,m)$ with $l<\log\log n$ and $0<m\le m_1$. This is because any subcomplex of type $(l,m)$ appears in $Y$ with probability $\left(c/n\right)^{l(d+2)+m}$.

The proof of Theorem 4.1  in ~\cite{ALLM} yields a constant $\alpha=\alpha(c,d)$ such that a.a.s. $\LMc$ has no inclusion-minimal subcomplex that is a core with $m_1\le m \le \alpha n^d$ $d$-faces. In fact, their argument only uses the fact that a minimal core $C$ is {\em connected} in the sense that between every two $(d-1)$-faces $\tau,\tau'$ in $C$ there is a path alternating between $(d-1)$-faces and $d$-faces of $C$ with an inclusion relation. However, since every core is a union of connected cores, this means that there are no cores of size $m$ in $Y$.
It follows that the only possible cores that $Y$ can contain have type $(l,0)$, i.e., it is a $d$-gravel.
\end{proof}
\subsection{The core of $\LMc$ - Theorem \ref{thm:main} (II.a)}
The proof of this theorem closely follows the argument of Riordan \cite{rio}. We fix the dimension $d$ and refer to $r(c):=1-t(c,d)$ as a function of $c$. For brevity we denote $r^+(c):= \Psi_2(cr(c)^d)$. Note that both $r(c)$ and $r^+(c)$ are continuous, bounded away from $0$ and increasing when $c>\ccol$. Our main goal is to show that for every $\tilde c>\ccol$ and $\varepsilon>0$, $f_{d-1}(\tilde Y) > (r^+(\tilde c)-\varepsilon)\binom nd$, where $\tilde Y$ is the $d$-core of $\LMvarc{\tilde c}$. 

This is motivated by the fact that $r^+(\tilde c)$ is the probability that the root's degree is $\ge 2$ after every finite number of rooted collapse phases in $T_d(\tilde c)$. In other words, this is the probability that the root survives the non-rooted collapse process. Although this argument is simple and appealing, the actual proof is substantially more involved. Our strategy is to define some carefully crafted property $\A$ of $d$-trees of depth $\log\log n \ll S=S(n) \ll \log n$ such that the following two statements hold a.a.s. First, the subset $A\subset Y_{d-1}$ of $(d-1)$-faces $\tau$ such that $Y$ contains a $\tau$-rooted $d$-tree with property $\A$ is of density at least $(r^+(\tilde c)-\varepsilon)$. Second, for every $\tau\in A$ there exists a $d$-tree $T_{\tau}\subset Y$ in which $\tau$'s degree is at least $2$ and every leaf also belongs to $A$. Consequently, no $(d-1)$-face in $A$ can be collapsed.

We refer throughout the proof to certain properties of rooted $d$-trees $(T,o)$, and occasionally write $T\in\P$ to say that $T$ has property $\P$. Every property $\P$ of rooted $d$-trees induces a property of $(d-1)$-faces in a $d$-complex $Y$. Namely, we say that the $(d-1)$-face {\em $\tau$ has property $\P$} if $Y$ contains a $d$-tree rooted on $\tau$ that has property $\P$. 

Here are some relevant properties: $\D_{\le L}$ means $T$ has depth $\le L$, and $\D_{<\infty}$ means that it has finite depth. Let $\P_1$ and $\P$ be properties of finite (resp. general) $d$-trees. A $d$-tree $(T,o)$ has property $\P_1\circ\P$ when: (i) $T$ has a finite subtree $T'$ rooted at $o$ with property $\P_1$, and (ii) For every leaf $\tau$ of $T'$, the subtree of $T$ rooted at $\tau$ has property $\P$. For example, we consider the property that $T$ has depth $k+1$ and it does not collapse in $k$ phases, i.e., $\B_k:= \{T\in \D_{\le k+1}~|~\delta_k(T) > 0\}$ and note that $\B_k = \B_0\circ\B_{k-1}$. Property $\B$ means that $T$ does not collapse at finite time.

We also define for $k\ge 0$, the properties $\R_k$ which are stronger than $\B_k$ as follows
\[
\R_0 = \{T\in \D_{\le 1}~|~\delta_0(T) \ge 2\}~,~~~\R_k = \R_0\circ\B_{k-1} \cup \B_0\circ\R_{k-1},~k>0.
\]
The difference between $\B_k$ and $\R_k$ is this: $T\in\B_k$ means that $T$ has depth $k+1$ and every non-leaf $(d-1)$-face has at least one descendant $d$-face. In defining $\R_k$ we add the requirement that along every root-to-leaf path we encounter at least one $(d-1)$-face with $\ge 2$ descendant $d$-faces.

Finally, we introduce a stochastic version of $\P$, a property of finite rooted $d$-trees $(T,o)$. For some $0\le p\le 1$ we {\em mark} each leaf of $T$ independently with probability $p$, and remove every $d$-face that contains any unmarked leaf. We say that the event $\M(\P,p,T)$ holds if the remaining $d$-tree has property $\P$. Marking is a convenient way of capturing the following phenomenon: We let each leaf in a finite $d$-tree grow a Poisson $d$-tree and we only ask whether or not this "tail" has some desired property. This is expressed in the simple identity
\begin{equation}
\label{eqn29}
\Pr[\M( \P_1,p,T_{d,k}(c) )]=\Pr[T_d(c)\in \P_1\circ\P]
\end{equation}
where $d,k$ are integers, $c>0$, $\P_1$ is a property of depth-$k$ trees and $\P$ is a property of probability $p$ for Poisson $d$-trees with parameter $c$. For instance,
\[\Pr[\M(\B_k, r(c), T_{d,k+1}(c))]=\Pr[T_d(c)\in\B]=r(c)\]

The following lemma can be viewed as a variation on this identity. It shows that although property $\R_k$ is stronger than $\B_k$, the two are almost equally likely in a Poisson $d$-tree.
\begin{lemma}
\label{lem:RKfrequent}
For every $c>c_1>\ccol$ there is a sufficiently large $k$ such that
\[
\Pr[\M(\R_k, r(c_1), T_{d,k+1}(c))]>r(c_1).
\]
\end{lemma}
 
\begin{proof}
Consider the following probabilistic experiment where we randomize thrice. Initially we generate the first $k+1$ generations of $T_d(c)$. Then we do the random marking that yields the $d$-tree $T$. Finally we remove each $d$-face of $T$ independently with probability $c_1/c$. We denote the component of the root by $T'$. Note that $T'$ is distributed like a $T_{d,k+1}(c_1)$ to which random $r(c_1)$-marking is applied. In particular, $\Pr[T'\in\B_k]=r(c_1)$, hence we need to prove that $\Pr[T\in\R_k]>\Pr[T'\in\B_k]$. Since $T'\in\B_k$ implies that $T\in\B_k$ it suffices to show that $\Pr[W]>\Pr[L]$, where
\[
W=[T\in\R_k,~T'\notin\B_k]~~\text{and}~~ L=[T\in\B_k\setminus\R_k,~T'\in\B_k].
\]
To this end we show that $\Pr[L]\to 0$ as $k\to\infty$ whereas $\Pr[W]$ stays bounded away from $0$.

Indeed, if $T\in\B_k\setminus\R_k$, then there exists some $d$-face of depth $k$ whose removal violates property $\B_k$. This $d$-face survives in $T'$ with probability $(c_1/c)^k$, so that $\Pr[L]<(c_1/c)^k$. On the other hand $W$ contains the event that $\delta_0(T)=2$, $T\in \R_0\circ\B_{k-1}$ and $\delta_0(T')=0$ whose probability is positive and independent of $k$.
\end{proof}

A $d$-tree $T$ of depth $L+1$ is {\em $(p,\eta)$-rigid} if $\Pr[\M(\R_L,p,T)]>1-\eta$.
\begin{lemma}
\label{lem:rigid}
For every $c>c_1>\ccol$ and $\eta>0$ and for a large enough integer $L$ there holds
\[
\Pr[T_{d,L+1}(c) \text{~is~} (r(c),\eta)\text{-rigid}] \ge r(c_1).
\]
\end{lemma}
\begin{proof}
Below we assume that $k$ is large enough, as required in Lemma \ref{lem:RKfrequent}. We claim that
\begin{equation*}
\label{eqn:17}
\Pr[T_d(c)\in \R_k\circ\B]=\Pr[\M( \R_k,r(c),T_{d,k+1}(c) )]\ge \Pr[\M(\R_k,r(c_1),T_{d,k+1}(c))]>r(c_1).
\end{equation*}
The equality is a special case of identity~(\ref{eqn29}). The next inequality follows by a simple monotonicity consideration and the fact that $r(c)>r(c_1)$. The last inequality comes from Lemma \ref{lem:RKfrequent}.
The condition of $(p,\eta)$-rigidity is stronger the smaller $\eta$ is, so we fix it to satisfy \[ \Pr[T_d(c)\in \R_k\circ\B]-r(c_1)>\eta.\] 

For fixed $k$ the conditions $T_{d,k+l+2}(c)\in \R_k\circ\B_l$ become more strict as $l$ grows and their conjunction over all $l\ge0$ is exactly the condition $T_d(c)\in \R_k\circ\B$. Therefore, we can and will choose $l$ large enough so that
$$\Pr[T_{d,k+l+2}(c)\in \R_k\circ\B_l]\le \Pr[T_d(c)\in\R_k\circ\B]+\eta^2.$$ 

For a $d$-tree $T$ of depth $L+1=k+l+2$, let $\phi(T):=\Pr[\M( \R_k\circ\B_l,r(c),T )]$. We denote by $\cL$ the property that $T$ is $(r(c),\eta)$-rigid.
The expectation of $\phi(T)$, where $T$ is $T_{d,L+1}(c)$ distributed, equals the probability $\Pr[T_d(c)\in\R_k\circ\B]$. In addition, $\phi(T)=0$ if $T\notin\R_k\circ\B_l$ and $\phi(T)\le 1-\eta$ if $T\notin\cL$ since $\R_k\circ\B_l$ implies $\R_L$. Therefore
\[
\Pr[T_d(c)\in\R_k\circ\B] \le \Pr[T_{d,L+1}(c)\in(\R_k\circ\B_l)\cap\cL]+(1-\eta)\Pr[T_{d,L+1}(c)\in(\R_k\circ\B_l)\setminus\cL],
\]
whence
\[
\eta\cdot\Pr[T_{d,L+1}(c)\in\R_k\circ\B_l\setminus\cL] \le \Pr[T_{d,L+1}(c)\in\R_k\circ\B_l]-\Pr[T_d(c)\in\R_k\circ\B].
\]

Putting everything together we conclude that $\Pr[T_{d,L+1}(c)\in\cL]\ge\Pr[T_{d,L+1}(c)\in\R_k\circ\B_l]-\eta>r(c_1)$, as stated.

\end{proof}

We set all the parameters that appear in the discussion below. Recall that our goal is to show that for every $\tilde c>\ccol$ and $\varepsilon>0$, $f_{d-1}(\tilde Y) > (r^+(\tilde c)-\varepsilon)\binom nd$. Let $\ccol<c_1<c<\tilde c$ such that $r^+(c_1)\ge r^+(\tilde c)-\varepsilon/2$ and $\Psi_1(\tilde c r(c_1)^d)>r(c)$. Again we choose $k$ large enough to make Lemma \ref{lem:RKfrequent} hold, and we fix some $0<\eta < d^{-2(k+1)}/8$. Also $L$ is so large that Lemma \ref{lem:rigid} holds. Recall that $\cL$ denotes the property $(r(c),\eta)$-rigidity of $d$-trees of depth $L+1$. 

Consider an integer $S=S(n)$ such that $\log\log n \ll S \ll \log n$, and we define for $0\le s\le S$, the properties $\A_s$ by
\[
\A_0 = \cL~,~~~\A_s = \R_k\circ\A_{s-1},~0<s\le S.
\]
Note that $\A_S$ depends on the first $Q:=(k+1)S+L+1$ generations. Finally, we define two key properties:
\[
\A = \R_0 \circ \D_{\le L}\circ \A_S~,~~~ \P=\R_0\circ\D_{<\infty}\circ\A.
\]
In words, $T\in\A$ means that $T$ has a rooted subtree of depth $\le L+1$ in which the root's degree is $\ge 2$ and the subtree of $T$ rooted at each of its leaves has property $\A_S$. $T\in\P$ means that $T$ has a finite rooted subtree in which the root's degree is $\ge 2$ and the subtree of $T$ rooted at each of its leaves has property $\A$.

It follows by induction that $\Pr[T_{d,(k+1)s+L+1}\in\A_s] > r(c_1)$ for every $s\ge 0$. Indeed, Lemma \ref{lem:rigid} yields the case $s=0$, and the inductive step follows from Lemma \ref{lem:RKfrequent}. Also, since $\R_0\circ\A_S$ implies $\A$, the probability that $T_d(c)$ has a rooted subtree that satisfies $\A$ is at least
\begin{equation}
\label{eqn:pr_of_A}
 \Pr[T_{d,Q+1}(c)\in\R_0\circ\A_S] \ge\Psi_2(cr(c_1)^d)\ge\Psi_2(c_1r(c_1)^d)=r^+(c_1).
\end{equation}
Indeed, the probability that in a $d$-face which contains the root all the subtrees rooted at a $(d-1)$-face at depth $1$ has $\A_S$ is at least $r(c_1)^d$. Hence, the number of such $d$-faces is Poi$(cr(c_1)^d)$ distributed.

We come to the most significant step in the proof: We show, by a first-moment argument, that a.a.s.\ every $(d-1)$-face in $Y=\LMvarc{\tilde c}$ that has property $\A$ also has $\P$. This implies that every such $(d-1)$-face $\tau$ is contained in a $d$-tree $T_\tau \subset Y$ in which (i) $\tau$ is of degree $\ge 2$, and (ii) every leaf of $T_\tau$ has property $\A$. Consequently, the union of these $d$-trees $\{T_\tau~|~\tau\mbox{~has property~}\A\}$ is contained in the core $\tilde Y$.

\begin{claim}
\label{clm:struc_As}
For every $s\ge 0$, and every $T\in\A_s$, 
\[
\Pr[\M(\B_{(k+1)s}\circ\R_{L},r(c),T)]\ge 1 - \frac{2^ {-2^s}}{4d^{2(k+1)}}.
\]
\end{claim}
\begin{proof}
By induction on $s$. Our definition of $\cL$ and the choice of $\eta$ yield the case $s=0$. Let $s\ge0$ be an integer, and $T\in\A_{s+1} = \R_k\circ\A_{s}$. Let $T'\subset T$ be a {\em minimal} $d$-tree of depth $k+1$ such that $T'\in\R_k$ and every subtree rooted at its leaves has the property $\A_{s}$. By a straightforward computation, $T'$ has exactly $2d^{k+1}$ leaves. By induction, after the marking process the subtree rooted at every leaf of $T'$ fails to have property $\B_{(k+1)s}\circ\R_{L}$ independently with probability at most $\frac{2^ {-2^s}}{4d^{2(k+1)}}$. Let us refer to such a leaf as {\em bad}, and remove from $T'$ every $d$-face that contains a bad leaf. Since initially $T'$ had the property $\R_k$, it now has property $\B_k$ unless at least $2$ $d$-faces were removed. But this can only occur if at least $2$ leaves are bad, an event of probability at most
\[
\binom{2d^{k+1}}2\left(\frac{2^ {-2^s}}{4d^{2(k+1)}}\right)^2 \le \frac{2^{-2^{s+1}}}{4d^{2(k+1)}}.
\]
Namely, the tree $T$ has the property $\B_k\circ\B_{(k+1)s}\circ\R_{L}=\B_{(k+1)(s+1)}\circ\R_{L}$  with the desired probability. 
\end{proof}
This leads to the following key lemma.
\begin{lemma}
\label{lem:goodisprefect}
For every fixed $(d-1)$-face $\tau$ in $Y=\LMvarc{\tilde c}$, the probability that $\tau$ has property $\A$ but does not have $\P$ is $o(n^{-d})$.
\end{lemma}
\begin{proof}
It is easy to show that with probability $o(n^{-d})$ the $(1+L+2Q)$-neighborhood of $\tau$ consists of at most $n^{1/3}$ vertices, and we condition on this event. Suppose that $\tau$ has the property $\A$, and consider $\tau\in T\subset Y$ such that $T\in\A$. In particular, there exists a $d$-tree $T'\subset T$ rooted at $\tau$ of depth at most $L+1$ such that $d_{T'}(\tau)= 2$ and every subtree $T_{\pi}''\subset T$ rooted at a leaf $\pi$ of $T'$ has property $\A_S$. Denote by $X\subset Y_{d-1}$ the union of the leaves of $T_{\pi}''$ over all the leaves $\pi$ of $T'$.

We now expose an additional subset of the $(Q+1)$-neighborhoods of the $(d-1)$-faces of $X$ with the following precaution. When we reach some $(d-1)$-face $\rho$ and query whether a $d$-face that contains it belongs to $Y$, we only consider $d$-faces of the form $v\rho$ where $v$ is a vertex that does not belong to $T$ nor did it appear in the exposing process upto the current query. In this manner, every $\rho\in X$ is the root of a $d$-tree $\tilde T_{\rho}\subset Y$ in which every $(d-1)$-face has at least $\mbox{Bin}(n-n^{1/3},\frac{\tilde c}n)$ descendants. Therefore, the probability that $\tilde T_{\rho} \in \B_0\circ\A_S$ is at least
\[
\Pr[T_{d,Q}(\tilde c) \in\B_0\circ\A_S]-o(1) \ge \Psi_1(\tilde cr(c_1)^d) - o(1) > r(c),
\]
and these events are independent over $\rho\in X$. Therefore we can consider the event $\tilde T_{\rho} \in \B_0\circ\A_S$ as an alternative for marking the leaves of $d$-trees $T_{\pi}''$ and plug it in Claim \ref{clm:struc_As}. Since the number of leaves of $T'$ is bounded, the claim implies that with probability $1-O\left(2^{-2^S}\right)=1-o(n^{-d})$, after the described exposure of the additional neighborhoods, all the subtrees rooted in these leaves have the property $\B_{(k+1)S}\circ\R_L\circ\B_0\circ\A_S$. But recall that $\R_L$ means that in every path from the root of the $d$-tree to a $(d-1)$-face of depth $L+1$, there is a $(d-1)$-face with at least two descendants. In other words, $\R_L\circ\B_0\circ\A_S$ implies $\D_{\le L}\circ \R_0\circ\D_{\le L}\circ\A_S=\D_{\le L}\circ\A$. In particular, all the leaves of $T'$ have the property $\D_{<\infty}\circ\A$, and since $d_{T'}(\tau)= 2$, it follows that $\tau$ has the property $\P$.
\end{proof}

We are now ready to prove the theorem.
\begin{proof}[Proof of Theorem \ref{thm:main} (II.a)]
Let $N_{\A}$ denote the number of $(d-1)$-faces in $Y=\LMvarc{\tilde c}$ that have property $\A$. By the previous discussion, we know that $f_{d-1}(\tilde Y) \ge N_{\A}$. 
We approximate the expectation of $N_{\A}$ by $\mbox{Pr}_{\tilde c}[\A]$, the probability that $T_d(\tilde c)$ has a rooted subtree that satisfies $\A$,
\[
\E[N_{\A}] =(\mbox{Pr}_{\tilde c}[\A]+o(1))\binom nd\ge r^+(c_1)\binom nd\ge(r^+(\tilde c)-\varepsilon/2)\binom nd
\]
by Equation (\ref{eqn:pr_of_A}). Since $\A$ depends only on the $O(S)$-neighborhood of the $(d-1)$-face, and two $(d-1)$-faces have non-disjoint neighborhoods with negligible probability, it follows that $\E[N_{\A}^2]=\E[N_{\A}]^2(1+o(1)).$ By the second moment method $\Pr[N_{\A} < (r^+(\tilde c)- \varepsilon)\binom nd] = o(1)$. 
The upper bound is much simpler. Let $N_k$ denote the number of $(d-1)$-faces that survive (= did not collapse nor become free) the first $k$ phases of collapse. Clearly, $f_{d-1}(\tilde Y)\le N_k$ for every $k$. Similarly to $N_{\A}$, this property depends on the $k$-neighborhood of a $(d-1)$-face and by the same argument as before, $N_k$ is concentrated around its expectation. The expectation of $N_k$ can be bounded by the Poisson $d$-tree as follows.
\[
\E[N_k] \le (\mbox{Pr}_{\tilde c}[\delta_{k-1}\ge 2]+o(1))\binom nd= 
(\Psi_2(\tilde c (1-t_{k-2})^d )+o(1)) \binom nd,
\]
and since $t_k\to t$, we obtain that $\Psi_2(\tilde c (1-t_{k-2})^d)$ tends to $r^+(\tilde c)$ as $k\to\infty$. 

We turn to prove that a.a.s.\ the number of $d$-faces in the core $f_d(\tilde Y) = r(\tilde c)^{d+1}\frac {\tilde c}n\binom n{d+1}(1+o(1))$. Let $M_{\A}$ denote the number of $d$-faces in $Y$ all of whose $(d-1)$-faces have property $\A$. Clearly, $f_d(\tilde Y)\ge M_{\A}$ since no $(d-1)$-face with property $\A$ is collapsed. In addition, since this is a local property it suffices, as before, to compute the expectation of $M_{\A}$. The probability that a $d$-simplex $\sigma$ belongs to $Y$ is $\frac{\tilde c}n$, and we can expose a subset $T$ of its neighborhood in the same careful fashion as done in the proof of Lemma \ref{lem:goodisprefect}. 
The probability that all the $d$-trees growing from $\sigma$'s $(d-1)$-faces have property $\A_S$ is at least $\Pr_{\tilde{c}}[\A_S]^{d+1} - o(1)> r(c_1)^{d+1}$. If this occurs, then every $(d-1)$-face $\tau\subset\sigma$ has property $\A$ by letting $\tau$ be the root of the $d$-tree $T\cup\{\sigma\}$. Consequently $\E[M_{\A}] \ge r(c_1)^{d+1}\frac {\tilde c}n\binom n{d+1}.$ The upper bound is proved similarly, by showing that the probability that a $d$-face survives the first $k$ collapse phases tends to $r(\tilde c)^{d+1}$ as $k$ grows.
\end{proof}

\subsection{C-shadow of $\LMc$}
\label{subsec:C-shad}
Here we prove the parts of Theorem \ref{thm:randShadow} that deal with the $C$-shadow of $\LMc$. Namely, we show that for $c<\ccol$, the C-shadow of $Y=\LMc$ has size $\Theta(n)$, and for $c>\ccol$ its size is
\[
|\SH_C(Y)| = \binom n{d+1}((1-t)^{d+1}+o(1)).
\]
Both statements follow directly from the previous proofs. Regarding the range $c<\ccol$, a simple second moment calculation shows that a.a.s.\ there are $\Theta(n)$ sets of $d+2$ vertices in $Y$ that span all but one of the $d$-faces in the boundary of a $(d+1)$-simplex. The missing $d$-face in every such configuration is obviously in the C-shadow. On the other hand, if the C-shadow is large, viz., $|\SH_C(Y)|\gg n$, then for every $c<c'<\ccol$, with probability bounded away from zero, the core of $\LMvarc{c'}$ contains a complex that is not the boundary of $(d+1)$-simplex. But this contradicts Theorem \ref{thm:main}(I).

We prove the supercritical case $c>\ccol$ in much the same way that we calculated the number of $d$-faces in the core. Namely, for the lower bound we count $d$-simplices not in $Y$ all of whose $(d-1)$-faces have property $\A$. For the upper bound we count $d$-simplices that if added into $Y$ do not survive $k$ phases of collapse. As before, both properties are local and by a second moment argument are concentrated around their means, which are computed by Poisson $d$-tree approximations.

\section{$d$-acyclicity}
In the previous section we saw that the threshold $\ccol$ for $d$-collapsibility in $\LMc$ coincides with the threshold in which rooted collapsibility in $T_d(c)$ almost surely eliminates all the $d$-faces containing the root. In the case of $d$-acyclicity, the correspondence is similar but more intricate. In fact, the threshold $\cacy$ for $d$-acyclicity coincides with {\em two} seemingly separate thresholds of $T_d(c)$'s parameters. These will used to bound the $d$-acyclicity threshold from below and above respectively. Since both occur at $c=\cacy$ it follows that these bounds are tight. Furthermore, if $c>\cacy$, these two parameters yield upper and lower bounds for $\beta_d(Y,\mR)$ which are tight upto small order error terms. Finally, the tight estimation for $\beta_d(Y,\mR)$ allows us to compute the density of the {\em shadow} of $Y$.

If a $d$-complex $Y$ has more $d$-faces than $(d-1)$-faces, then $\beta_d(Y,\mR)\ge f_d(Y)-f_{d-1}(Y)>0$. For $Y=\LMc$, this happens only when $c>d+1$, but we can say a bit more. Even though $Y$ and its core $\tilde Y$ have the same $d$-th Betti number, it turns out that there is a wider range of the parameter $c$ for which $\tilde Y$ has more $d$-faces than $(d-1)$-faces. In fact, one can show that $f_d(\tilde Y)>f_{d-1}(\tilde Y)$ if and only if $c>\cacy$, using the expressions for these face numbers in Theorem \ref{thm:main}(II.a). However, it is significantly easier to prove the same lower bound on $\beta_d(Y)$ by analyzing $T_d(c)$ as follows. Let $S_k(Y)$ be obtained by removing all exposed $(d-1)$-faces in $R_k(Y)$. The average degree of the $(d-1)$ faces in $S_k(Y)$ is approximated using the conditional expectation
$$
\E[\delta_k ~|~ \delta_k >0~\wedge~\delta_{k-1}>1].
$$
In words, this is the expected degree of the root of $T_d(c)$  after $k$ phases of rooted collapses, conditioned on the fact that its degree remains strictly greater than $1$ throughout the collapse process. We claim that this captures the average degree of $(d-1)$-faces in $S_k(Y)$. A $(d-1)$-face $\tau$ of $Y$ belongs to $S_k(Y)$ if and only if its degree after $k-1$ phases of $\tau$-rooted collapse is $>1$ and stays positive after one more phase. Indeed, as long as $d_\tau>1$ the $\tau$-rooted collapse and non-rooted collapse are identical. A difference occurs when $d_\tau = 1$, at which point the rooted collapse continues as usual, but the non-rooted collapse eliminates $\tau$. If the average degree of $(d-1)$-faces exceeds $d+1$, this yields, via a simple double-counting argument, a positive lower bound for $\beta_d(Y)$. A simple calculus exercise then shows that this condition holds if and only if $c>\cacy$. Namely,
\[
\lim_{k\to\infty}\E[\delta_k ~|~ \delta_k >0~\wedge~\delta_{k-1}>1] > d+1 ~~\iff ~~ c > \cacy.
\]

The most substantial role of local weak convergence is in proving the lower bound on the $d$-acyclicity threshold. We analyze $T=T_d(c)$ using tools from spectral theory and functional analysis. For this reason we are still unable to resolve this question over finite fields of coefficients. As further detailed below, we define $x_T:=\pi_{L(T),e_o}(\{0\})$ to be the measure of the atom $\{0\}$ according to the {\em spectral measure $\pi$} of the {\em Laplacian $L(T)$} of $T$ with respect to the characteristic vector $e_o$ of the root $o$. Local weak convergence implies that {\em $x_T$ is an upper bound on the normalized dimension} of the left kernel $Z$ of $\partial_d(Y)$. Note that if $\LMc$ is $d$-acyclic then $\dim Z$ equals $(1+o(1)){n\choose d}\left(1-\frac c{d+1}\right)$, and otherwise it is greater. Indeed, the proof shows that 
\[
\E_{T\sim T_d(c)}[x_T] = 1 - \frac{c}{d+1}~~\iff~~c<\cacy,
\]
and if $c>\cacy$, this expectation is greater than $1 - \frac{c}{d+1}.$

\label{sec:acy}
\subsection{Acyclicity beyond collapsibility  - Theorem \ref{thm:main}(II.b)}
To prove that a random complex is $d$-acyclic beyond the $d$-collapsibility threshold, we cannot restrict ourselves to purely combinatorial arguments. It is not a-priori clear that the local weak limit of a random complex holds enough information to prove such a statement. Surprisingly, perhaps, this is the case when we work over $\mathbb R$. In this section we describe the main ingredients of this method, which appears in ~\cite{LP}, where complete proofs can be found.

Let $Y=\LMc$. The primary goal in the proof is to find a tight upper bound for $\lim_{n\to\infty}\frac{1}{{n\choose d}}\E[\beta_d(Y)]$. It turns out more useful to work with the corresponding Laplace operator $L(Y)=\partial_d(Y)\partial_d(Y)^*$. We consider its kernel $Z$ which coincides with the left kernel of $\partial_d(Y)$. Let $P_Z:\mR^{Y_{d-1}}\to Z$ be the orthogonal projection to the space $Z$. By linear algebra,
\[
\dim Z = \sum_{\tau\in Y_{d-1}} \| P_Z(e_\tau) \|^2,
\]
where $e_\tau$ is the unit vector of $\tau$.

The {\em spectral theorem} from functional analysis offers a new perspective of $\| P_Z(e_\tau) \|^2$. Associated with every self-adjoint operator $L$ on a Hilbert space $\mathcal H$, and a vector $\psi\in\mathcal H$ is the {\em spectral measure of $L$ with respect to $\psi$}. It is a real measure denoted $\pi_{L,\psi}$ which satisfies
\[
\inpr{F(L)\psi,\psi}=\int_\mR F(x) d\pi_{L,\psi}(x),
\]
for every measurable function $F:\mR\to\C$. The operator $F(L)$ is uniquely defined by extending the action of polynomials on the operator $L$.

If $\mathcal H$ is finite-dimensional, $\pi_{L,\psi}$ is a discrete measure supported on the spectrum of $L$ and $\pi_{L,\psi}(\lambda)=\|P_{\lambda}\psi\|^2$, where $P_\lambda$ is the orthogonal projection to the $\lambda$-eigenspace.

We use this theorem with the measure $\pi_{L(Y),e_\tau}$. Here $Y$ is a $d$-complex and $\mathcal H = \ell^2(Y_{d-1})$. The self adjoint operator is the Laplacian $L(Y)$, and $e_\tau$ is the characteristic vector of some $(d-1)$-face of $Y$.

In particular, with $Y=\LMc$ and $Z$ as before, $\| P_Z(e_\tau) \|^2$ is simply the measure of the atom $\{0\}$ according to the spectral measure $\pi_{L(Y),e_\tau}$.

The difficulty with applying the spectral theorem to the Poisson $d$-tree is that the degrees in this tree may be unbounded. We must, therefore, consider the subtleties of the theory of unbounded operators ~\cite{book_unbounded}. Briefly, the Laplacian $L(T)$ of an infinite $d$-tree $T$ is a symmetric operator, directly defined on the dense subset of finitely supported vectors of $\mathcal H=\ell^2(T_{d-1}).$ The symmetric densely-defined operator $L(T)$ has a unique extension to $\mathcal H$. This extension need not be self-adjoint, and when it does we say that the tree $T$ is self-adjoint. In such cases the spectral theorem can be applied on $L(T)$. It can be shown that a Poisson $d$-tree is, almost surely, self-adjoint.

We employ the useful property that spectral measures are continuous with respect to local weak convergence. Since $Y=\LMc$ converges in local weak convergence to the Poisson $d$-tree $T=T_d(c)$, which is almost-surely self-adjoint, we conclude that the expected measure $\E_Y[\pi_{L(Y),e_\tau}]$ weakly converges to the expected measure $\E_T[\pi_{L(T),e_o}]$, where $o$ is the root of $T$. In particular, by measuring the closed set $\{0\}$,
\[
\limsup_{n\to\infty}\E\left[\|P_Z(e_\tau)\|^2\right] \le \E[x_T],
\]
where $x_T:=\pi_{L(T),e_o}(\{0\})$.
Consequently,
\[
\E[\dim Z] \le (1+o_n(1)){n\choose d} \E[x_T].
\]
By the Rank-Nullity Theorem from linear algebra,
\[
\dim Z - \beta_d(Y) = f_{d-1}(Y) - f_d(Y),
\]
and we conclude that
\[
\frac{1}{{n\choose d}}\E[\beta_d(Y)]\le\E[x_T] - 1 + \frac{c}{d+1} + o_n(1).
\]

There remains the problem of bounding the expectation $\E_{T}\left[x_T\right]$ without directly computing the operator's kernel. This difficulty is bypassed using the recursive structure of $d$-trees to derive a simple recursion formulas on these spectral measures, as in the following lemma.

Let $T$ be a self-adjoint $d$-tree with root $o$, and let $\sigma_1,...,\sigma_m$ be the $d$-faces that contain the root.  For $1\le j\le m$ and $1 \le r \le d$ we denote by $\tau_{j,r}$ the $(d-1)$-faces of $\sigma_j$ other than the root. Also, $T_{j,r}$ denotes the $d$-tree rooted at $\tau_{j,r}$ that contains $\tau_{j,r}$ and its branch.
\begin{lemma}
\label{lem:recTree}
 $x_T=0$ if there exists some $1\le j \le m$ such that $x_{T_{j,1}}=...=x_{T_{j,d}}=0$. Otherwise,
\[
x_T = \left(1+\sum_{j=1}^{m}\left(\sum_{r=1}^{d}x_{T_{j,r}} \right)^{-1} \right)^{-1}.
\]
\end{lemma}
\begin{proof}[Proof sketch.] 
Consider the following  bounded family of measurable functions $\{H_s:\mR\to\C~|~s\in\mR\setminus\{0\}\},$
\[
H_s(x) = \frac{is}{x+is}.
\]
Note that $H_s$ approaches the Kronecker delta function $\delta_{x,0}$ as $s\to 0$. Given a self-adjoint $d$-tree $T$ with root $o$, we define
$$h_T(s):=\int_\mR H_s(x)d\pi_{L(T),e_o}(x).
$$
In particular, $x_T = \lim_{s\to 0}h_T(s)$.
The proof is concluded by showing that the functions $h_{T_{j,r}}$'s and $h_T$ satisfy the formula
\begin{equation}
\label{eqn:h_formulas}
h_T(s)\left( 1+\sum_{j=1}^{m}\left(is+\sum_{r=1}^{d}h_{T_{j,r}}(s)\right)^{-1} \right) = 1,
\end{equation}
and letting $s\to 0$.

We turn to describe the derivation of equation (\ref{eqn:h_formulas}). Denote $L:=L(T)$, $\tl{L}:=\bigoplus_{j,r}L(T_{j,r})$ and $M$ the Laplacian of the subcomplex of $T$ which contains $\sigma_1,...,\sigma_m$ and their subfaces. In particular, $L=M\oplus\tl L$. The operators $H_s(L)$ and    $H_s(\tl L)$ are scalar multiples of the {\em resolvents} $R:=(L+	is\cdot I)^{-1}$ and $\tl R:=(\tl L+	is\cdot I)^{-1}$. In particular, $h_T(s) = is\cdot\inpr{Re_0,e_0}$. Simple observations about the $d$-tree structure yields that (i) $\tl R e_o = \frac {1}{is} e_o$ and (ii) $\inpr{\tl R e_{\tau_{j,r}},e_{\tau_{j',r'}}}=0$ when $(j,r)\ne (j',r').$ In addition, we use the Second Resolvent Identity which states that $RM\tl R = \tl R - R$. Here the operator $M$ has a very concrete and usable form, being the Laplacian of a $d$-complex which consists of $m$ distinct $d$-faces with a common $(d-1)$-subface $o$. Equation (\ref{eqn:h_formulas}) is obtained using simple algebraic manipulations by comparing terms of the form $\inpr{(RM\tl R)e_\tau,e_{\tau'} }= \inpr{(\tl R - R)e_\tau,e_{\tau'}}$, when $\tau,\tau'$ are $(d-1)$-faces of $T$ of distance at most $1$ from the root $o$.
\end{proof}

The remaining step of the argument is an application of the recursive formula in Lemma \ref{lem:recTree} to the Poisson $d$-tree.
\begin{lemma}
\label{lem::recPoiTree}
Let $T=T_d(c)$ be a rooted Poisson $d$-tree with parameter $c$. Then,
\[
\E[x_T ]\le \max\left\{t+ct(1-t)^d-\frac{c}{d+1}\left(1-(1-t)^{d+1}\right)~\mid~ t\in\I,~t=e^{-c(1-t)^d}\right\}
\]
\end{lemma}
Note that this maximum is taken over a finite set, due to the condition $t=e^{-c(1-t)^d}$.
\begin{proof}
Let $T$ be a Poisson $d$-tree with root degree $m$ and $\{T_{j,r}~|~1\le j\le m,1\le r\le d\}$ its subtrees as above. The parameters $x_T,\{x_{T_{j,r}}\}$ can be considered as random variables when $T$ is $T_d(c)$-distributed. The random variables $\{x_{T_{j,r}}\}$ are i.i.d and are distributed like $x_T$ since all the subtrees $T_{j,r}$ are independent Poisson $d$-trees. In addition, these variables satisfy the equation of Lemma \ref{lem:recTree}.

These observations suggest the following equivalent description of $\mathcal D$, the distribution of the random variable $x_T$. First sample a $Poi(c)$-distributed integer $m$, and $x_{T_{j,r}}\sim \mathcal D$ i.i.d\ for every $1\le j\le m$ and $1\le r\le d$. Given these samples, the value of $x_T$ is determined by Lemma \ref{lem:recTree}.

In particular, if we let $t:=\Pr(x_T>0)$, then $t$ satisfies the equation
\begin{equation}\label{eqn:p_eqn}
t = \sum_{m=0}^{\infty}\frac{e^{-c}c^m}{m!}(1-(1-t)^d)^m=e^{-c(1-t)^d}.
\end{equation}
Let $X$ be a $\mathcal D$-distributed random variable,
\[
\E[X]~=~\E\left[\frac{\textbf{1}_{\{\forall j\in [m],\;S_j>0\}}}{ 1+  \sum_{j=1}^{m}S_j^{-1} }\right]
\]
Here $S_1,S_2,\ldots,S_m$ are random variables whose distribution is that of a sum of $d$ i.i.d.\ $\mathcal D$-distributed variables. By expressing the probability $\Pr[\textbf{1}_{\{\forall j\in [m],\;S_j>0\}}]$ as $t$, exploiting the symmetry between the different $S_j$'s and using basic properties of the Poisson distribution, we are able to express this expectation in terms of $t$, 
\[
\E[X] = t+ct(1-t)^d-\frac{c}{d+1}\left(1-(1-t)^{d+1}\right),
\]
as was claimed.
\end{proof}

It requires only basic calculus to conclude that:
\begin{enumerate}
\item For $c<\cacy$, the maximum of $$t + ct(1-t)^d-\frac{c}{d+1}\left(1-(1-t)^{d+1}\right),~~\mbox{s.t.}~~~t=e^{-c(1-t)^d}$$ is attained at $t=1$. Consequently, $\E[\beta_d(Y)]\le o(n^d).$
\item For $c>\cacy$, the maximum is attained at $t=t(c,d)$. Consequently,
\begin{equation}
\E[\beta_d(Y)]\le (1+o_n(1)){n\choose d}\left( \frac{c(1-t)^{d+1}}{d+1} - 1+t+ct(1-t)^d \right).
\label{eqn:beta_upper_bound}
\end{equation}

\end{enumerate}

The proof of Theorem \ref{thm:main}~(II.b) is concluded by the following standard probabilistic argument. Let $c<\cacy$ and $\varepsilon > 0$. The absence of small non-collapsible subcomplexes in $Y$ (Lemma \ref{lem:noSmallCores}) implies that it has no small $d$-cycles except $\partial\Delta_{d+1}$'s. The computation above shows that the dimension of the cycle space is $o(n^d)$. Therefore, removing $\mbox{Bin}\left({n\choose d+1},\varepsilon/n\right)=\Theta(n^d)$ random $d$-faces eliminates all large (non $\partial\Delta_{d+1}$) $d$-cycles with high probability, hence $\LMvarc{c-\varepsilon}$ is a.a.s.\ $d$-acyclic except $\partial\Delta_{d+1}$'s.

\subsection{The cyclic regime - Theorem \ref{thm:main}(III)} 

The key idea of ~\cite{acyc} for the computation of a matching upper bound for the $d$-acyclicity threshold uses an analysis of the collapse process. Let $Y$ be a $d$-complex and $k$ some positive integer. Recall that $R_k(Y)$ is the simplicial complex obtained from $Y$ by $k$ phases of $d$-collapse and $S_k(Y)$ is its subcomplex obtained by removing the exposed $(d-1)$-faces. Clearly, $\beta_d(Y)=\beta_d(S_k(Y))$ since $d$-collapsing and removing exposed faces does not affect the right kernel of $\partial_d$. The final ingredient of the strategy is the observation that $\beta_d(S_k(Y)) \ge f_d(S_k(Y)) - f_{d-1}(S_k(Y)),$ and the fact that the parameter $f_d(S_k(Y)) - f_{d-1}(S_k(Y))$ can be studied from the local weak limit.

Let $Y=\LMc$ where $c>\cacy$, and $Y':=S_k(Y)$ for a sufficiently large integer $k$.

\begin{align}
\E[\beta_d(Y)]~\ge~&\E[f_d(Y')-f_{d-1}(Y')]\nonumber\\
=~&\E\left[ \sum_{\tau \in {Y'}_{d-1}}\left(\frac{d_{Y'}(\tau)}{d+1}-1\right) \right]\nonumber\\
=~&{n\choose d} \Pr[\tau\in {Y'}]\left(\frac{\E\left[ {d_{Y'}(\tau)|\tau\in {Y'}}\right]}{d+1}-1\right)\label{eqn:beta_lower_bound}.
\end{align}

For the last equation we can, due to symmetry, consider a fixed $\tau$ and apply the Law of Total Expectation to the event $\{\tau\in {Y'}\}$.
As mentioned above, the $(d-1)$-face $\tau$ of $Y$ belongs to $Y'$ if and only if in the $\tau$-rooted collapse process of $Y$, the degree of $\tau$ is greater than $1$ after $(k-1)$ phases and positive after $k$ steps. By approximating the $\tau$-rooted collapse process of $Y$ with the rooted collapse process on $T_d(c)$ we obtain that upto $1+o_n(1)$ factor,
\begin{equation}
\label{eqn:tau_in_Y}
\Pr[\tau\in Y'] \ge \Pr[\delta_k>1]=1-t_k- c(1-t_{k-2})^dt_{k-1}.
\end{equation}
Furthermore, upto $1+o_n(1)$ factor,
\begin{align}
\E\left[ {d_{Y'}(\tau)|\tau\in {Y'}}\right]~=~&\E[\delta_k~|~\delta_k>0~\wedge~ \delta_{k-1}>1]\nonumber\\
\ge~&\sum_{j=2}^{\infty} j\cdot \Pr[\delta_k=j~|~\delta_k>0~\wedge~ \delta_{k-1}>1]\nonumber\\
\ge~&\sum_{j=2}^{\infty} j\cdot \frac{\Pr[\delta_k=j]}{\Pr[\delta_{k-1}>1]}\nonumber\\
=~&\frac{c(1-t_{k-1})^{d}(1-t_k)}{ 1-t_{k-1}- c(1-t_{k-2})^dt_{k-1}}.\label{eqn:expected_degree}
\end{align}
By combining Inequalities (\ref{eqn:beta_lower_bound}),(\ref{eqn:tau_in_Y}) and (\ref{eqn:expected_degree}), and letting $k\to\infty$, we obtain that
\[
\E[\beta_d(Y)]\ge (1+o_n(1)){n\choose d}\left( \frac{c(1-t)^{d+1}}{d+1} - 1+t+ct(1-t)^d \right).
\]
This bound starts to be meaningful for $c>\cacy$, where this expression matches the upper bound from (\ref{eqn:beta_upper_bound}). Since $\beta_d$ is 1-Lifschitz, a straightforward application of Azuma's inequality yields that a.a.s.\ $\beta_d$ deviates from its expectation by only $o(n^d)$. In particular, this shows a matching upper bound for the threshold of $d$-acyclicity.

\subsection{$\mR$-Shadow of $\LMc$}
The behavior of the $\mR$-shadow when $c<\cacy$ is studied similarly to the C-shadow in the collapsible regime (See Section \ref{subsec:C-shad}). 

We turn to the range $c>\cacy$. Here we do not give a proof, but only a general intuitive explanation. An accurate analysis of the measure concentration can be found in \cite{LP}.
Recall that $$\frac{1}{{n\choose d}} \E[\beta_d(Y;\mR)]
\xrightarrow[n\to\infty]{}g_d(c):=\frac{c}{d+1}(1-t)^{d+1}-(1-t)+ct(1-t)^d.$$
A simple technical claim shows that for every $c>\cacy$, the limit function $g_d(c)$ is differentiable w.r.t the variable $c$ and its derivative equals to $\frac{1}{d+1}(1-t)^{d+1}$.

It turns out to be more convenient to work here with a $d$-dimensional analog of the so-called {\em evolution of random graphs}. Let $Y_d(n,m)$ be a random simplicial complex with $n$ vertices, a complete $(d-1)$-skeleton and $m$ uniformly random $d$-faces. $Y'=Y_d(n,m+1)$ can be sampled by the following procedure. First sample $Y=Y_d(n,m)$ and then add a random $d$-face which does not belong to $Y$. Therefore, the following equation holds in expectation,
\[
\beta_d(Y';\mR) - \beta_d(Y;\mR) = \frac{1}{{n\choose d+1}}\left|\SH_\mR(Y)\right|.
\]
Letting $m\sim\mbox{Bin}\left({n \choose {d+1}},\frac cn\right)$ yields
\[
{n\choose d}\left(g_d\left(c+\frac{d+1}{{n\choose d}}\right) - g_d(c)\right) \approx \frac{1}{{n\choose d+1}}\left|\SH_\mR\left(Y\right)\right|,
\]
and by letting $n\to\infty$,
\[
\frac{1}{{n\choose d+1}}\left|\SH_\mR\left(Y\right)\right|\approx (d+1)g'_d(c) = (1-t)^{d+1}.
\]
This argument can be made rigorous by incrementing $Y$ with $\varepsilon n^d$ random $d$-faces at a time rather than one by one, and applying standard measure concentration inequalities.

\section{Concluding remarks}
Although this article is mostly a review of previous work, it does contain several new results, e.g., Theorem \ref{thm:main}(I) that deals with the collapsible regime is a little stronger than the original result of \cite{ALLM}. Other notable new results concern the asymptotic densities of the core and the C-shadow of $\LMc$ for $c>\ccol$, improving the main theorem of \cite{col2} which says that $\LMc$ is a.a.s.\ non-collapsible for $c>\ccol$. As mentioned above, although the main result in that paper is correct, there is an error in the proof, which we are able to remedy here using the techniques of Riordan~\cite{rio}.

The results surveyed here can be viewed from several perspectives which suggest different problems for future research.

From the combinatorial perspective, the phase transition in the density of the shadow of $\LMc$ is of great interest. We conjecture that the $\mR$-shadow grows from linear in $n$ to a giant (order $\Theta(n^{d+1})$) in a single step in a random evolution of simplicial complexes. This starkly contrasts with the gradual growth of the giant component in random graphs. It is of particular interest to understand the structure of the critical complex. Numerical experiments suggest that its $(d-1)$-homolgy group has torsion of size $\exp(\Theta(n^d))$, but we do not know a proof of this yet.

On the topological side, it would be very interesting to better understand the $d$-cycles which appear in $\LMc$ when $c>\cacy$. Provably, they consist of $\Theta(n^d)$ $d$-faces, but numerical experiments suggest that they lie in an unknown territory in the realm of homological $d$-cycles. Unlike closed manifolds, in which the degree of all $(d-1)$-faces equals to $2$, it seems that in these $d$-cycles the average degree approaches $d+1$, which is the largest possible for a minimal $d$-cycle. Namely, these $d$-cycles are in some sense the opposite of manifolds.

In addition, the random complexes $\LMc$, where $\ccol<c<\cacy$ have the nice property of being $d$-acyclic but not $d$-collapsible. What other interesting topological or combinatorial properties do they have?

There is much more to study about random simplicial complexes in the regime $p=c/n$. In particular, the question regarding the vanishing of the top homology over finite fields is still open and presently out of reach. It is interesting to resolve whether homology thresholds over other fields can also be read off some parameter of the Poisson $d$-tree.

\label{sec:open}


\begin{thebibliography}{10}
\bibitem{ald_zeta}
David Aldous.
\newblock The $\zeta(2)$ limit in the random assignment problem.
\newblock {\em Random Structures \& Algorithms}, 18(4):381--418, 2001.

\bibitem{ald_lyo}
David Aldous and Russell Lyons.
\newblock Processes on unimodular random networks.
\newblock {\em Electron. J. Probab}, 12(54):1454--1508, 2007.


\bibitem{ald_ste}
David Aldous and Michael Steele.
\newblock The objective method: Probabilistic combinatorial optimization and
  local weak convergence.
\newblock In {\em Probability on discrete structures}, pages 1--72. Springer,
  2004.

\bibitem{col2}
Lior Aronshtam and Nathan Linial.
\newblock The threshold for collapsibility in random complexes.
\newblock {\em arXiv preprint arXiv:1307.2684}, 2013.

\bibitem{acyc}
Lior Aronshtam and Nathan Linial.
\newblock When does the top homology of a random simplicial complex vanish?
\newblock {\em Random Structures \& Algorithms}, 2013.

\bibitem{ALLM}
Lior Aronshtam, Nathan Linial, Tomasz {\L{}}uczak, and Roy Meshulam.
\newblock Collapsibility and vanishing of top homology in random simplicial
  complexes.
\newblock {\em Discrete \& Computational Geometry}, 49(2):317--334, 2013.


\bibitem{BHK}
Eric Babson, Christopher Hoffman, and Matthew Kahle.
\newblock The fundamental group of random 2-complexes.
\newblock {\em Journal of the American Mathematical Society}, 24(1):1--28,
  2011.

\bibitem{ben_sch}
Itai Benjamini and Oded Schramm.
\newblock Recurrence of distributional limits of finite planar graphs.
\newblock In {\em Selected Works of Oded Schramm}, pages 533--545. Springer,
  2011.


%

\bibitem{diluted}
Charles Bordenave, Marc Lelarge, and Justin Salez.
\newblock The rank of diluted random graphs.
\newblock {\em Annals of Probability}, 39(3):1097--1121, 2011.


%


\bibitem{HKP}
Christopher Hoffman, Matthew Kahle, and Elliot Paquette.
\newblock The threshold for integer homology in random d-complexes.
\newblock {\em arXiv preprint arXiv:1308.6232}, 2013.

%
%
\bibitem{KPS}
D{\'a}niel Kor{\'a}ndi, Yuval Peled, and Benny Sudakov.
\newblock A random triadic process.
\newblock {\em SIAM Journal on Discrete Mathematics}, 30(1):1--19, 2016.

\bibitem{LM}
Nathan Linial and Roy Meshulam.
\newblock Homological connectivity of random 2-complexes.
\newblock {\em Combinatorica}, 26(4):475--487, 2006.


\bibitem{LNPR}
Nathan Linial, Ilan Newman, Yuval Peled, and Yuri Rabinovich.
\newblock Extremal problems on shadows and hypercuts in simplicial complexes.
\newblock {\em arXiv preprint arXiv:1408.0602}, 2014.


\bibitem{LP}
Nathan Linial and Yuval Peled.
\newblock On the phase transition in random simplicial complexes.
\newblock {\em Annals of Mathematics, accepted for publication}, 2016.

\bibitem{LuP}
Tomasz \L{}uczak and Yuval Peled.
\newblock Integral homology of random simplicial complexes.
\newblock {\em arXiv preprint arXiv:1607.06985}, 2016.

\bibitem{lyo_asym}
Russell Lyons.
\newblock Asymptotic enumeration of spanning trees.
\newblock {\em Combinatorics, Probability and Computing}, 14(04):491--522,
  2005.


\bibitem{MW}
Roy Meshulam and Nathan Wallach.
\newblock Homological connectivity of random k-dimensional complexes.
\newblock {\em Random Structures \& Algorithms}, 34(3):408--417, 2009.

\bibitem{molloy}
Michael Molloy.
\newblock Cores in random hypergraphs and boolean formulas.
\newblock {\em Random Structures \& Algorithms}, 27(1):124--135, 2005.

\bibitem{book_unbounded}
Michael Reed and Barry Simon.
\newblock {\em Methods of modern mathematical physics: Functional analysis},
  volume~1.
\newblock Gulf Professional Publishing, 1980.

\bibitem{rio}
Oliver Riordan.
\newblock The k-core and branching processes
\newblock {\em Combinatorics, Probability \& Computing}, 17(1):111--136, 2008.

\end{thebibliography}
\end{document}